\documentclass[12pt]{amsart}
\usepackage{amssymb}
\usepackage{fullpage}
\usepackage{url}
\usepackage{enumerate}
\usepackage{eucal}

\theoremstyle{plain}
\newtheorem{theorem}{Theorem}

\newtheorem{proposition}[theorem]{Proposition}
\newtheorem{corollary}[theorem]{Corollary}

\theoremstyle{definition}
\newtheorem{definition}[theorem]{Definition}
\newtheorem{definitions}[theorem]{Definitions}
\newtheorem{remark}[theorem]{Remark}
\newtheorem{remarks}[theorem]{Remarks}
\newtheorem{example}[theorem]{Example}
\newtheorem{notation}[theorem]{Notation}

\newcommand{\normgen}[1]{{\lhd}{#1}{\rhd}}
\newcommand{\ZZ}{{\mathbb Z}}
\newcommand{\GG}{{\mathfrak G}}
\newcommand{\RR}{{\mathcal R}}
\newcommand{\WW}{{\mathfrak W}}
\renewcommand{\leq}{\leqslant}
\newcommand{\Gmma}{\mathit{\Gamma}}
\renewcommand{\geq}{\geqslant}
\renewcommand{\epsilon}{\varepsilon}
\newcommand{\cprime}{$'$}
\newcommand{\nasmof}{nasmof}

\def\d1{\discretionary{-}{}{-}}

\begin{document}

\title{Left relatively convex subgroups}

\author{Yago Antol\'{\i}n}
\address{Department of Mathematics, Vanderbilt University, 1326 Stevenson Center, Nashville, TN 37240, USA}
\email{yago.anpi@gmail.com}
\thanks{The first- and second-named authors' research was partially supported by Spain's Ministerio de Ciencia e
Innovaci\'on through Project MTM2011-25955}

\author{Warren Dicks}
\address{Departament de Matem{\`a}tiques, Universitat Aut{\`o}noma de Barcelona, E-08193 Bellaterra (Barcelona), SPAIN}
\email{dicks@mat.uab.cat}

\author{Zoran {\v{S}}uni{\'c}}
\address{Department of Mathematics, Texas A\&M University, College Station, TX 77843-3368, USA}
\email{sunic@math.tamu.edu}
\thanks{The third-named author's research was partially supported by the National Science Foundation under Grant
No.~DMS-1105520}


\begin{abstract}
Let $G$ be a group and $H$ be a subgroup of $G$.
We say that  $H$  is left relatively convex in
$G$ if the left $G$-set $G/H$ has at least one $G$-invariant order; when $G$ is left orderable,
this holds if and only if $H$ is convex in $G$ under some left  ordering of $G$.

We give a criterion for $H$ to be left relatively convex in $G$ that generalizes a famous theorem of
Burns and Hale and has essentially the same proof. We show that all maximal cyclic subgroups are left relatively convex
in free groups, in right-angled Artin groups, and in surface groups that are not
the Klein-bottle group. The free-group case
 extends a result of Duncan and Howie.

We show that if $G$ is left orderable, then each free factor of $G$
 is left relatively convex in~$G$.  More generally, for any graph of groups,
 if  each edge group is left relatively convex
in each of its vertex groups, then each
vertex group  is left relatively convex in the fundamental group;
this generalizes a   result of Chiswell.

We show that all maximal cyclic subgroups in locally residually torsion-free nilpotent groups are
left relatively convex.
\end{abstract}

\maketitle

\section{Outline}\label{sec:outline}

\begin{notation}
Throughout this article, let $G$ be a multiplicative group, and  $G_0$ be a subgroup  of $G$. For
$x$, $y \in G$, $[x,y]:= x^{-1}y^{-1}xy$, \,\, $x^y:= y^{-1}xy$,\,\, and\,\, \mbox{$\null^yx :=
yxy^{-1}$}. For any subset $X$ of~$G$, \mbox{$X^{\pm 1}:= X \cup X^{-1}$},   $\langle X \rangle$
denotes the subgroup of $G$ generated by $X$,  $\langle X^G \rangle$ denotes the normal subgroup of
$G$ generated by $X$, and \mbox{$G/\normgen{X}:= G/\langle X^G \rangle$}. When we write $A
\subseteq B$ we mean that $A$ is a subset of $B$, and when we write $A \subset B$ we mean that $A$
is a proper subset of $B$.
\end{notation}

In Section~\ref{sec:relconv}, we collect together some facts, several of which first arose in  the
proof of  Theorem 28 of~\cite{bergman:ordered}.  If $G$ is left orderable, Bergman calls $G_0$
`left relatively convex in~$G$' if $G_0$ is convex in~$G$ under some left ordering of $G$, or,
equivalently,  the left $G$-set $G/G_0$ has some $G$-invariant order. Broadening the scope of his
terminology, we shall say that   $G_0$   is \textit{left relatively convex in}~$G$ if the left
$G$-set $G/G_0$ has some   $G$-invariant order, even if $G$ is not left orderable.

We give a criterion for $G_0$ to be left relatively convex in $G$
that generalizes  a famous theorem of Burns and Hale~\cite{BH72}
and has essentially the same proof.  We deduce  that if each noncyclic,  finitely
generated subgroup of $G$  maps onto $\ZZ^2$,  then each
 maximal  cyclic subgroup of $G$ is
left relatively convex in $G$.  Thus, if $F$ is a  free group and $C$ is a
 maximal  cyclic subgroup of~$F$, then
$F/C$ has an  $F$-invariant order;  this extends the result of Duncan and Howie~\cite{DH91}
that a certain finite subset of $F/C$  has an order that is respected by the partial $F$-action.
Louder and Wilton~\cite{LW14} used the Duncan-Howie order to prove Wise's
conjecture that, for subgroups $H$ and $K$ of a free group $F$, if $H$ or $K$ is
a maximal cyclic subgroup of $F$, then \mbox{$\sum_{HxK \in H\backslash\,F/K}
 \operatorname{rank}(H^x \cap   K) \leq \operatorname{rank}(H)\operatorname{rank}(K)$}.
They also gave a simple proof of the existence of a Duncan-Howie order;
translating their argument from topological to algebraic  language  led us to the order on $F/C$.

In Section~\ref{sec:graphs}, we find that the main result of~\cite{DS14}
 implies that, for any graph of groups, if  each edge group is left relatively convex
in each of its vertex groups, then each
vertex group   is left relatively convex in the fundamental group.
This generalizes a result of
Chiswell~\cite{chiswell:ordered}.
In particular, in a left-orderable group, each free factor is left relatively convex.

One says that $G$ is   \textit{discretely left orderable} if some
infinite  (maximal) cyclic subgroup of~$G$ is left relatively convex in $G$.
Many examples of such groups are given in \cite{LRR09};  for instance, it is seen that
among  free groups, braid  groups,  surface
groups,  and  right\d1angled Artin groups, all the infinite ones are
 discretely left orderable. In Section~\ref{sec:examples} below, we show that all maximal  cyclic subgroups are
left relatively convex in right\d1angled Artin groups
and in  surface groups that are not the Klein-bottle group.

At the end, in Section~\ref{sec:rtfn}, we show that all maximal cyclic subgroups in locally
residually torsion-free nilpotent groups are left relatively convex.

\section{Left relatively convex subgroups}\label{sec:relconv}

\begin{definitions}\label{defs:order}
Let $X$ be a set and $\RR $ be a binary relation on  $X$; thus, $\RR $ is a subset of \mbox{$X \times X$}, and
 \mbox{`$x\RR y$'} means \mbox{`$(x,y) \in \RR $'}.  We say that  $\RR $ is \textit{transitive}
when, for all \mbox{$x$, $y$, $z \in X$}, if   \mbox{$x \RR y$}   and
 \mbox{$y \RR  z$}, then  \mbox{$x \RR  z$}, and here
we write  \mbox{$x \RR  y\RR  z$} and say that $y$ \textit{fits between $x$ and $z$ with respect to~$\RR $}.
We say that  $\RR $ is \textit{trichotomous} when, for
all $x$, $y \in X,$    exactly one of  \mbox{$x \RR y$}, \mbox{$x=y$}, and  \mbox{$y \RR x$}  holds, and here
we say that the \textit{sign} of the triple  \mbox{$(x, \RR ,y)$}, denoted
 \mbox{$\operatorname{sign}(x, \RR ,y)$},
is 1, 0, or $-1$, respectively.
A transitive,  trichotomous  binary relation is called an \textit{order}.
 For any order $<$ on $X$, a subset $Y$ of $X$ is said to be \textit{convex in $X$  with respect to~$<$}
 if no element of  \mbox{$X{-}Y$} fits between two elements of $Y$ with respect to $<$.

Now suppose that $X$ is a left $G$-set. The diagonal left $G$-action on \mbox{$X \times X$} gives
a left $G$-action on the set of binary relations on $X$.
By a \textit{binary $G$-relation on~$X$} we mean a $G$-invariant binary relation on $X$,
and  by  a   \textit{$G$-order on~$X$} we mean a $G$-invariant order on~$X$.
 If there exists at least one  $G$-order on $X$,
we say that $X$ is \textit{$G$-orderable}.
If $X$ is endowed with a  $G$-order,
 we say that $X$ is \textit{$G$-ordered}.
When $X$ is $G$  with the left multiplication action, we replace `$G$-' with `left', and write
\textit{left order}, \textit{left orderable}, or  \textit{left ordered}, the latter two being hyphenated when they
premodify a noun.

Analogous terminology applies for right $G$-sets.
\end{definitions}

\begin{definitions}\label{defs:chain}  For $K \leq H \leq G$, we  recall two
mutually inverse operations.  Let \mbox{$x$, $y \in G$}.

If $<$  is a    $G$-order on $G/K$ with respect to which $H/K$ is convex in~$G/K$,
then we define an  $H$-order $<_{\operatorname{bottom}}$ on  $H/K$ and a
$G$-order  $<_{\operatorname{top}}$  on  $G/H$ as follows.  We take $<_{\operatorname{bottom}}$ to be
 the restriction of $<$ to $H/K$.  We define
\mbox{$xH <_{\operatorname{top}} yH$}  to mean
 \mbox{$(\forall h_1,\,h_2 \in H)(xh_1K < yh_2K)$}.
This relation is trichotomous since \mbox{$xH <_{\operatorname{top}} yH$} if and only if
\mbox{$(xH \ne yH) \wedge (xK < yK)$}; the former clearly implies the latter, and, when the latter holds,
 \mbox{$K  < x^{-1}yK$},  and then, by the convexity of $H/K$ in~$G/K$,
 \mbox{$h_1K < x^{-1}yK$},
and then  \mbox{$y^{-1}xh_1K < K$},\, \mbox{$y^{-1}xh_1K < h_2K$},\,
and \mbox{$xh_1K <  yh_2K$}.
Thus, $<_{\operatorname{top}}$ is a   $G$-order on $G/H$.

Conversely, if
$<_{\operatorname{bottom}}$ is an  $H$-order on $H/K$ and
$<_{\operatorname{top}}$ is a $G$-order on $G/H$, we now define
a $G$-order $<$ on $G/K$ with respect to which $H/K$ is convex in $G/K$.
We define \mbox{$xK < yK$} to mean \vspace{-3mm} \newline\centerline{$(xH <_{\operatorname{top}}  yH) \,\, \vee\,\,
\bigl((xH = yH) \wedge (K <_{\operatorname{bottom}} x^{-1}yK)\bigr).$}\vspace{1mm}\newline   It is clear that $<$ is a
well-defined   $G$-order on $G/K$.  Suppose \mbox{$xK \in  (G/K)-(H/K)$}. Then
\mbox{$xH \ne H$}. If
\mbox{$xH <_{\operatorname{top}}  H$}, then \mbox{$xK < hK$}, for all
\mbox{$h \in H$}, and similarly if  \mbox{$H <_{\operatorname{top}}  xH$}.  Thus,
$H/K$ is convex in $G/K$   with respect to $<$.

In particular, $G/K$ has some $G$-order with respect to which $H/K$ is  convex in~$G/K$
if and only if $H/K$ is $H$-orderable and $G/H$ is $G$-orderable.
Taking $K = \{1\}$ and $H=G_0$, we find that the following are equivalent,
as seen in the proof of Theorem~28 (vii)$\Leftrightarrow$(viii) of \cite{bergman:ordered}.
\begin{enumerate}[\normalfont(\ref{defs:chain}.1)]
\item\label{a2} $G$ has some left order with respect to which $G_0$ is  convex in $G$.
\item\label{c2} $G_0$ is left orderable, and $G/G_0$ is $G$-orderable.
\item\label{b2} $G\mkern7mu$ is left orderable, and $G/G_0$ is $G$-orderable.
\end{enumerate}
\end{definitions}

 This  motivates  the  terminology introduced in the following definition, which presents an analysis similar to
one given by Bergman in the proof of  Theorem~28  in \cite{bergman:ordered}. Unlike Bergman, we do not require that the group $G$ is left-ordered.

\begin{definition}  \label{def:relconvex}
Let $\operatorname{Ssg}(G)$ denote the set of all the subsemigroups of $G$, that is, subsets of~$G$
closed under the multiplication. We say that the subgroup \textit{$G_0$ of $G$ is  left relatively
convex in} $G$ when any of the following equivalent conditions  hold.
\begin{enumerate}[\normalfont(\ref{def:relconvex}.1)]
\item\label{a1} The left $G$-set $G/G_0$ is  $G$-orderable.
\item\label{b1} The right $G$-set $G_0\backslash G$ is  $G$-orderable.
\item\label{c1} There exists some \mbox{$G_+ \in \operatorname{Ssg}(G)$} such that
    \mbox{$G_+^{\pm1} = G{-}G_0$}; in this event, \mbox{$G_+ \cap G_+^{-1} = \emptyset$}  and
    \mbox{$G_0G_+ = G_+G_0 = G_0G_+G_0 = G_+$}.
\item\label{d1} For each finite subset $X$
of \mbox{$G{-}G_0$}, there exists
\mbox{$S \hskip-2pt\in\hskip-2pt  \operatorname{Ssg}(G)$} such that \mbox{$ X \hskip-2pt \subseteq \hskip-2pt  S^{\pm 1}
\hskip-2pt  \subseteq  \hskip-2pt G {-} G_0 $}.
\end{enumerate}
We then say also   that $G_0$
\textit{is a left relatively convex subgroup of}~$G$.  One may also use `right'  in place of `left'.
\end{definition}

\begin{proof}[Proof of   equivalence] (\ref{def:relconvex}.\ref{a1})$\Rightarrow$(\ref{def:relconvex}.\ref{c1}).
 Let $<$ be a $G$-order on $G/G_0$, and set
$$ G_+:= \{x \in G \mid G_0 < xG_0\}; $$  then $G_+^{-1} = \{x \in G \mid G_0 < x^{-1}G_0\}
= \{x \in G \mid xG_0 <  G_0\}$ and $G_0 = \{x \in G \mid G_0 = xG_0\}.$  Hence,  $G_+^{\pm1} = G{-}G_0$.
If $x$, $y \in G_+$, then
$G_0 < xG_0$, $G_0 < yG_0$ and $G_0 < xG_0 < xyG_0$; thus $xy \in G_+$.  Hence,  $G_+ \in \operatorname{Ssg}(G)$.

 Now consider any $G_+ \in \operatorname{Ssg}(G)$ such that
\mbox{$G_+^{\pm1} = G{-}G_0$}.  Then $G_+ \cap G_+^{-1} = \emptyset$, since $G_+$ is a subsemigroup which does
not contain $1$.  Also, $G_0G_+ \cap G_0 = \emptyset$,  since $G_+ \cap G_0^{-1}G_0 = \emptyset$,
while  \mbox{$G_0G_+ \cap G_+^{-1} = \emptyset$},  since $G_0 \cap G_+^{-1} G_+^{-1} = \emptyset$.
Thus $G_0G_+ \subseteq G_+$, and equality must hold.  Similarly, $G_+G_0 = G_+$.

(\ref{def:relconvex}.\ref{c1})$\Rightarrow$(\ref{def:relconvex}.\ref{a1}).   Let
 \mbox{$x$, $y$, $z \in G$}.  We define  $xG_0 < yG_0$
to mean \mbox{$(xG_0)^{-1}(yG_0) \subseteq G_+$}, or, equivalently, $x^{-1}y \in G_+$.
Then $<$ is a well-defined binary $G$-relation on $G/G_0$.
 Since $x^{-1} y$
belongs to exactly one of $G_+$, $G_0$, and $G_+^{-1}$, we see that $<$ is trichotomous.  If
 \mbox{$xG_0 <  yG_0$} and \mbox{$yG_0 <  zG_0$}, then $G_+$ contains $x^{-1}y$, $y^{-1}z$,
and their product,  which shows that
\mbox{$xG_0 < zG_0$}.   Thus $<$ is a $G$-order on $G/G_0$.

(\ref{def:relconvex}.\ref{b1})$\Leftrightarrow$(\ref{def:relconvex}.\ref{c1})  is the
left-right dual of~(\ref{def:relconvex}.\ref{a1})$\Leftrightarrow$(\ref{def:relconvex}.\ref{c1}).

(\ref{def:relconvex}.\ref{c1})$\Rightarrow$(\ref{def:relconvex}.\ref{d1}) with $S = G_+$.

(\ref{def:relconvex}.\ref{d1})$\Rightarrow$(\ref{def:relconvex}.\ref{c1}).~Berg\-man~\cite{bergman:ordered}
observes that an  implication of this type follows easily from
the Com\-pact\-ness Theorem of Model Theory; here, one could equally well use
the  quasi\d1com\-pact\-ness of  $\{-1,1\}^{G-G_0}$, which holds by
 a famous theorem of Tychon\-off~\cite{Tych30}.
The case of this implication where \mbox{$G_0 = \{1\}$}  was first stated by Conrad~\cite{Con59}, who
gave a short argument designed to be read in conjunction with a short argument of Ohnishi~\cite{Ohn52}.
Let us show that a streamlined form of the Conrad-Ohnishi proof gives the general case comparatively easily.

Let \mbox{$2^{\mkern1mu G-G_0}$}  denote  the set of all subsets of~\mbox{$G{-}G_0$}.
  For  each \mbox{$W \in 2^{\mkern1mu G-G_0}$},
let  \mbox{$\operatorname{Fin}(W)$}  denote the set of finite subsets of~$W$, and
\mbox{$\langle\langle W \rangle \rangle$}
denote the subsemigroup of $G$ generated by $W$. For each \mbox{$\varphi \in \{-1,1\}^{G-G_0}$} and \mbox{$x \in G{-}G_0$},
set \mbox{$\tilde\varphi(x):= x^{\varphi(x)} \in \{x,x^{-1}\}$}.  Set\vspace{-2.5mm}
\begin{align*}
 \WW: = \biggl\{W \mkern-3mu \in 2^{G-G_0}  \mid
 \bigl( \forall\, W' & \mkern-3mu \in  \operatorname{Fin}(W)\bigr) \,\,
\bigl( \forall  X   \mkern-3mu \in  \operatorname{Fin}(G{-}G_0)\bigr )\,\,
\\[-5mm]
 & \bigl(\exists \varphi \in \{-1,1\}^{G-G_0} \bigr)
\Bigl(G_0 \cap \bigl\langle\bigl\langle \,\,W' \cup \tilde\varphi(X)  \,\,
\bigr\rangle \bigr\rangle   = \emptyset\Bigr)\,\,\biggr \}.\\[-8.5mm]
\end{align*}
It is not difficult to see that~(\ref{def:relconvex}.\ref{d1}) says precisely that  \mbox{$\emptyset \in \WW$}.  Also,
it is clear that\newline
\centerline{$\textstyle (\forall \,W\mkern-3mu \in 2^{\mkern1mu  G-G_0}) \Bigl(\bigl(W\mkern-3mu \in \WW\bigr)\Leftrightarrow\bigl(
\operatorname{Fin}(W) \subseteq \WW\bigr)\Bigr).$}\newline
It follows that $\WW$ is closed in  $2^{\mkern1mu G-G_0}$ under the
 operation of taking unions of chains. By Zorn's Lemma, there exists some maximal element $W$ of~$\WW$.

We shall prove that \mbox{$\langle \langle W \rangle \rangle^{\pm 1} =  G{-}G_0$}, and
thus~(\ref{def:relconvex}.\ref{c1}) holds.
By taking \mbox{$X = \emptyset$} in the definition of `\mbox{$W \mkern-3mu\in \WW$}', we see that
 \mbox{$ \langle \langle W \rangle \rangle  \subseteq G{-}G_0$}, and thus
 \mbox{$W^{\pm1} \subseteq \langle \langle W \rangle \rangle^{\pm 1}  \subseteq G{-}G_0$}.
It remains to show that \mbox{$G{-}G_0 \subseteq W^{\pm 1}$}.  Since $W$ is maximal in~$\WW$, it suffices to show that
\newline\centerline{\mbox{$(\forall x \in G{-}G_0) \bigl( (W \cup \{x\}  \in
\WW)\vee(W \cup \{x^{-1}\} \in \WW)\bigr)$.}}\newline
Suppose then    \mbox{$W \cup \{x\}\mkern-3mu \not\in \WW$}; thus, we may  fix
\mbox{$W_x \mkern-3mu\in \operatorname{Fin}(W)$} and  \mbox{$X_x \mkern-3mu\in \operatorname{Fin}(G{-}G_0)$} such
that\newline
\centerline{$\textstyle  \bigl(\forall \varphi \in \{-1,1\}^{G-G_0}  \bigr)\,\,
 \Bigl(G_0 \,\,\cap\,\, \bigl\langle \bigl\langle\,\, W_x \cup \{x\} \cup \tilde\varphi(X_x)\,\,
 \bigr\rangle\bigr\rangle  \ne \emptyset\Bigr).$}\newline
Let \mbox{$W'\mkern-3mu \in \operatorname{Fin}(W)$} and \mbox{$X\mkern-3mu \in \operatorname{Fin}(G{-}G_0)$}.
As \mbox{$W \mkern-3mu \in \WW$},
there exists
 \mbox{$\varphi \in \{-1,1\}^{G-G_0}$}    such that
\newline
\centerline{$G_0 \cap \bigl\langle \bigl\langle\,\,  W_{x} \cup W' \cup \tilde\varphi(\{x\} \cup X_{x} \cup X) \,\,
  \bigr\rangle \bigr\rangle   = \emptyset.$}\newline
Clearly, \mbox{$\tilde\varphi(x) \ne x$}.  Thus, \mbox{$\tilde\varphi(x)= x^{-1}$} and
\newline
\centerline{$G_0 \cap \bigl\langle \bigl\langle  \,\,  W' \cup \{x^{-1}\}\cup \tilde\varphi(X) \,\,
\bigr\rangle \bigr\rangle   = \emptyset.$}\newline   This shows that \mbox{$W \cup \{x^{-1}\} \mkern-3mu \in \WW$},
 as desired.
 \end{proof}

The Burns-Hale theorem~\cite[Theorem 2]{BH72} says that if each nontrivial, finitely generated subgroup of $G$ maps onto some nontrivial, left-orderable group, then $G$ is left orderable. The following result, using a streamlined version of their proof, generalizes the  Burns-Hale theorem in two ways. Namely, the scope is increased by stating the result for an arbitrary subgroup $G_0$ (in their case $G_0$ is trivial) and by imposing a weaker condition (in their case $\langle X \rangle$ is required to map onto a left-orderable group). 

\begin{theorem}\label{th:BH} If, for each nonempty, finite subset $X$ of $G{-}G_0$,
there exists a proper,
left  rela- tively convex  subgroup of  \mbox{$\langle X \rangle$}
that includes \mbox{$\langle X \rangle \cap\mkern2muG_0$}, then\, $G_0$~is left relatively convex in $G$.
\end{theorem}

\begin{proof} For each finite subset $X$ of  \mbox{$G{-}G_0$}, we shall  construct an
 \mbox{$S_X \in \operatorname{Ssg}(\langle X \rangle)$} such that
 \mbox{$X  \subseteq S_X^{\pm 1} \subseteq G{-}G_0$}, and then~(\ref{def:relconvex}.\ref{d1}) above will hold.
We set \mbox{$S_{\emptyset}:= \emptyset$}.  We now assume that \mbox{$X \ne \emptyset$}.  Let us
write \mbox{$H:= \langle X \rangle$}.  By  hypothesis, we have an $H_0$ such that \mbox{$H \cap G_0
\leq H_0 < H$} and $H_0$ is left relatively convex in $H$. Notice that \mbox{$H{-}H_0\subseteq
H{-}(H\cap G_0) \subseteq  G{-}G_0$} and \mbox{$X \cap H_0   \subset X$}, since \mbox{$X \nsubseteq
H_0$}. By induction on \mbox{$\vert X \vert$}, we have an \mbox{$S_{X\cap H_0} \in
\operatorname{Ssg}(\langle X \cap H_0 \rangle)$}  such that \mbox{$X \cap H_0   \subseteq S_{X\cap
H_0}^{\pm 1} \subseteq G{-}G_0$}. By~(\ref{def:relconvex}.\ref{c1}) above, since $H_0$ is left
relatively convex in $H$, we have an \mbox{$H_+ \in \operatorname{Ssg}(H)$} such that
 \mbox{$H_0 H_+ H_0 = H_+$}  and \mbox{$H_+^{\pm1} = H{-}H_0$}.
We set  \mbox{$S_X:=  S_{X\cap H_0}  \cup H_+$}. Then \mbox{$S_X \in \operatorname{Ssg}(H)$}, since
\mbox{$S_{X \cap H_0} \subseteq H_0$} and \mbox{$H_0H_+H_0 =H_+$}. Also, \newline
\vspace{-.5cm}\centerline{
 \mbox{$X =  (X \cap H_0) \cup (X{-}H_0)  \subseteq   S_{X\cap H_0}^{\pm 1}
\cup (H{-}H_0) = S_X^{\pm 1} \subseteq G{-}G_0.$}}
\end{proof}

\begin{remark}  Theorem~\ref{th:BH} above has a variety of corollaries.  For example,
 for any subset $X$ of  \mbox{$G$},  we have a sequence of successively weaker conditions:
\mbox{$\langle  X    \cup G_0 \rangle/\normgen{\mkern1muG_0\mkern1mu}$} maps onto~$\ZZ$;
\mbox{$\langle  X    \cup G_0 \rangle/\normgen{\mkern1muG_0\mkern1mu}$} maps onto a nontrivial, left-orderable group;
there exists a proper, left  relatively convex   sub\-group  of
\mbox{$\langle X \cup G_0  \rangle$} that includes \mbox{$G_0$}; and,
there exists a proper, left  relatively convex  subgroup of  \mbox{$\langle X \rangle$}
that includes \mbox{$\langle X \rangle \cap\mkern2muG_0$}. The last implication follows from the following fact. If $A$ and $B$ are subgroups of $G$ and $A$ is left relatively convex in $G$, then $A \cap B$ is left relatively convex in $B$. 
\end{remark}

\begin{definition}
A group $G$ is said to be $n$-\emph{indicable}, where $n$ is a positive integer, if it can be
generated by fewer than $n$ elements or it admits a surjective homomorphism onto $\ZZ^n$.

A group $G$ is \emph{locally $n$-indicable} if every finitely generated subgroup of $G$ is
$n$-indicable.
\end{definition}

Note that some authors require in the definition of indicability that $G$ admits a surjective
homomorphism onto $\ZZ$, while here 1-indicable means that $G$ is trivial or maps onto $\ZZ$,
2-indicable means that $G$ is cyclic or maps onto $\ZZ^2$, and so on.

\begin{example}
Free abelian groups of any rank and free groups of any rank are locally $n$-indicable for every
$n$.
\end{example}

The notion of $n$-indicability is related to left relative convexity through the following
corollary of Theorem~\ref{th:BH}.

\begin{corollary}\label{cor:BH}
Let $n \geq 2$. If $G$ is locally $n$-indicable group then each maximal $(n-1)$-generated subgroup
of $G$ is left relatively convex in $G$.

In particular, in a free group, each maximal cyclic subgroup is left relatively convex.
\end{corollary}

\begin{proof}
If the subgroup $G_0$ is maximal $(n-1)$-generated subgroup of $G$, then, for any nonempty, finite
subset $X$ of $G{-}G_0$, \mbox{$\langle  X  \cup G_0 \rangle$} maps onto \mbox{$\ZZ^n$}, and
\mbox{$\langle  X  \cup G_0 \rangle/\normgen{\mkern1muG_0\mkern1mu}$} maps onto~$\ZZ$.
\end{proof}

The idea of Corollary~\ref{cor:BH} can be used to show that certain maximal abelian subgroups are left relatively convex.

\begin{definition}
A group $G$ is \textit{\nasmof} if it is torsion-free and  every nonabelian finitely generated subgroup of $G$ admits a surjective homomorphism onto $\mathbb{Z}*\mathbb{Z}$. 
\end{definition}

\begin{example}
The class of {\nasmof} groups contains free and free abelian groups and it is closed under taking subgroups and direct products. 
Residually {\nasmof} groups are {\nasmof}, and in particular residually free groups are {\nasmof}. 
Every {\nasmof} group $G$ is $2$-locally indicable, and by Corollary \ref{cor:BH},  maximal
cyclic subgroups are left relatively convex. 
\end{example}

\begin{corollary}\label{cor:nasmof}
Let $n$ be a non-negative integer. If $G$ is a {\nasmof} group then each  maximal $n$-generated abelian subgroup of $G$ is left relatively convex in $G$.

In particular, in a residually free group, each maximal $n$-generated abelian subgroup is left relatively
convex.
\end{corollary}
\begin{proof}
If the subgroup $G_0$ is maximal $n$-generated abelian subgroup of $G$, then, for any nonempty, finite
subset $X$ of $G{-}G_0$, either \mbox{$\langle  X  \cup G_0 \rangle$} is  finitely generated, torsion-free abelian group of rank greater than $n$ or \mbox{$\langle  X  \cup G_0 \rangle$}  maps onto \mbox{$\ZZ*\ZZ$}. In both cases, \mbox{$\langle  X  \cup G_0 \rangle/\normgen{\mkern1muG_0\mkern1mu}$} maps onto~$\ZZ$.
\end{proof}

\section{Graphs of groups}\label{sec:graphs}

\begin{definitions}\label{defs:trees0} \label{defs:trees}
By a  \textit{graph}, we mean a quadruple $( \Gmma , V, \iota, \tau)$
such that $\Gmma$ is a set,  $V$ is a subset of $\Gmma$, and
$\iota$ and $\tau$ are  maps from
\mbox{$\Gmma  {-}\, V$} to $V$.
Here, we let $\Gmma$ denote the graph as well as the set, and  we write
 \mbox{$\operatorname{V}\mkern-4mu \Gmma := V$} and \mbox{$\operatorname{E}\mkern-1mu \Gmma:= \Gmma  {-}\, V$}, called the
\textit{vertex-set} and \textit{edge-set}, respectively.
We then define  \textit{vertex}, \textit{edge}
\mbox{$\iota e \xrightarrow{e}\mkern-9mu- \tau e$},
\textit{inverse edge} \mbox{$\tau e \xrightarrow{e^{-1}}\mkern-9mu- \iota e$},  \textit{path}  \vspace{1mm} \newline
(\ref{defs:trees0}.1) \hskip1.1cm  $v_0
\xrightarrow{e_1^{\epsilon_1}}{\mkern-9mu-} \,  v_1
 \xrightarrow{e_2^{\epsilon_2}}{\mkern-9mu-} \,   v_2
\xrightarrow{e_3^{\epsilon_3}}{\mkern-9mu-} \, \cdots
\xrightarrow{e_{n-2}^{\epsilon_{n-2}}}{\mkern-9mu-} \,
v_{n-2} \xrightarrow{e_{n-1}^{\epsilon_{n-1}}}{\mkern-9mu-} \,  v_{n-1}
 \xrightarrow{e_n^{\epsilon_n}}{\mkern-9mu-} \,   v_n,\,\, n \geq 0,$ \vspace{2mm} \newline
 \textit{reduced path}, and \textit{connected graph}  in the usual way.
We say that $\Gmma$ is a
 \textit{tree} if \mbox{$V \ne \emptyset$} and, for each $(v,w) \in V \times V$, there exists a
unique reduced  path  from $v$ to~$w$. The \textit{barycentric subdivision of}  $\Gmma$ is the
graph  $\Gmma^{(')}$ such that \mbox{$\operatorname{V}\mkern-4mu\Gmma^{(')} = \Gmma$} and
\mbox{$\operatorname{E}\Gmma^{(')} = \operatorname{E}\mkern-1mu\Gmma \times \{\iota, \tau\}$}, with
\mbox{$e \xrightarrow{(e,\iota)}{\mkern-13mu}-\iota e$} and
\mbox{$e \xrightarrow{(e,\tau)}{\mkern-13mu}-\tau e$}.

We say that $\Gmma$ is a \textit{left $G$-graph} if $\Gmma$ is a left $G$-set,
 $V$  is a $G$-subset of $\Gmma$, and $\iota$ and $\tau$ are $G$-maps.
For \mbox{$\gamma \in \Gmma$}, we let \mbox{$G_\gamma$} denote the $G$-stabilizer of $\gamma$.

Let $T$ be a tree.  A  \textit{local order on}  $T$ is a family  $(<_v \,\,\mid v \in \operatorname{V}\mkern-2mu T)$
such that, for each $v \in \operatorname{V}\mkern-2mu T$, $<_v$ is an order on
$\operatorname{link}_T(v) := \{e \in  \operatorname{E}  T \mid v \in \{\iota e, \tau e\} \}$.
By Theorem~3 of~\cite{DS14}, for each local order
 $(<_v \,\,\mid v \in \operatorname{V}\mkern-2mu T)$  on $T$,
 there exists a unique order $<_{\null_{\mkern2mu\scriptstyle{T}}}$
on~$\operatorname{V}\mkern-2mu T$   such that,  for each reduced $T$-path expressed as
in (\ref{defs:trees0}.1) above,
\newline \centerline{
$\textstyle{\operatorname{sign}(v_0,<_{\null_{\mkern2mu\scriptstyle{T}}}, v_n) =
\operatorname{sign}\bigl(0, \,\,<_{\null_{\mkern2mu\scriptstyle \ZZ}}, \,\,\sum\limits_{i=1}^n \epsilon_i +
\sum\limits_{i=1}^{n-1} \operatorname{sign}( e_{i} , <_{v_i} ,e_{i+1})\bigr),}$}\vspace{1mm}\newline
 where the sign notation is as in Definitions~\ref{defs:order} above.
We then call $<_{\null_{\mkern2mu\scriptstyle{T}}}$  the \textit{associated order},
\mbox{$\sum_{i=1}^n \epsilon_i$}  the \textit{orienta\-tion\d1sum},  and
\mbox{$\sum_{i=1}^{n-1} \operatorname{sign}( e_{i} , <_{v_i} ,e_{i+1})$} the
\textit{turn-sum}.  If $T$ is a left $G$-tree, then, for any  $G$-invariant local order on $T$,
  the associated order on $\operatorname{V}\mkern-3mu T$
is easily seen to be  a $G$-order.
\end{definitions}

\begin{theorem}\label{th:tree} Suppose that\, $T$ is a left $G$-tree such that, for each $T\mkern-3mu$-edge~\,$e$,
 $G_e$ is left relatively convex in  $G_{\iota e}$ and in $G_{\tau e}$.
Then, for each   \mbox{$t \in T$}, $G_t$ is left relatively convex in~$G$.
If there exists some  \mbox{$t \in T$} such that $G_t$ is left orderable, then $G$ is
left orderable.
Moreover, if the input orders are given effectively, then the output orders are  given effectively,
\end{theorem}

\begin{proof}   We choose one representative from each $G$-orbit in $\operatorname{V}\mkern-2mu T$.
For each representative~$v_0$, we choose an arbitrary order on the set of $G_{v_0}$-orbits $G_{v_0} \backslash
\operatorname{link}_T(v_0)$,
and, within each $G_{v_0}$-orbit, we choose one representative   $e_0$ and
a $G_{v_0}$-order on $G_{v_0}/G_{e_0}$, which exists by~(\ref{def:relconvex}.\ref{a1}) above;
since our $G_{v_0}$-orbit
$G_{v_0}e_0$ may be identified with $G_{v_0}/G_{e_0}$,
we then have a
 $G_{v_0}$-order on $G_{v_0}e_0$, and then on all of $\operatorname{link}_T({v_0})$ by our order
on  $G_{v_0} \backslash \operatorname{link}_T({v_0})$.  We then use    $G$-translates to
obtain a $G$-invariant local   order  on  $T$.  This in turn gives
the associated   $G$-order on $\operatorname{V}\mkern-2mu T$ as in  Definitions~\ref{defs:trees} above.
 In particular, for each $T\mkern-3mu$-vertex
$v$, we have   $G$-orders on $Gv$ and  $G/G_v$.   By~(\ref{def:relconvex}.\ref{a1}) above,
$G_v$ is then left relatively convex in $G$.  For each $T\mkern-3mu$-edge~\,$e$,  $G_e$ is
left relatively convex in $G_{\iota e}$ by
hypothesis, and then $G_e$ is left relatively convex in $G$ by Definitions~\ref{defs:chain} above.
 Thus, for each \mbox{$t \in T$}, $G_t$ is
left relatively convex in $G$.

By~(\ref{defs:chain}.\ref{c2})$\Rightarrow$(\ref{defs:chain}.\ref{b2}) above, if there exists some
\mbox{$t \in T$} such that $G_t$ is left orderable, then $G$ is
left orderable.
\end{proof}

\begin{example}\label{ex:trees} Let $F$ be a free group and $X$ be a free-generating set of $F$. The left Cayley graph   of
$F$ with respect to $X$ is a left $F$-tree on which $F$ acts freely. Thus, the fact that free groups are left orderable can
be deduced from  Theorem~\ref{th:tree} above; see~\cite{DS14}.
\end{example}

\begin{definitions}
 By a \textit{graph of groups} $(\GG , \Gmma)$, we mean a graph with vertex-set a family of groups
 \mbox{$(\GG (v')\mid v'  \in \operatorname{V}\Gmma^{(')})$}
and edge-set a family of injective group homomorphisms
\mbox{$(\GG (e) \xrightarrow{\GG (e')} \GG (v) \mid e\xrightarrow{e'}\mkern-9mu-v
 \in \operatorname{E}\Gmma^{(')})$}, where $\Gmma$ is a nonempty, connected graph and $\Gmma^{(')}$ is its barycentric subdivision.
For $\gamma  \in \Gmma^{(')}$, we call $\GG (\gamma)$ a \textit{vertex group}, \textit{edge group}, or \textit{edge map} if
$\gamma$ belongs to \mbox{$\operatorname{V}\mkern-4mu \Gmma$}, \mbox{$\operatorname{E}\mkern-1mu \Gmma$}, or
\mbox{$\operatorname{E}\mkern-1mu \Gmma^{(')}$}, respectively.
One may think of $(\GG , \Gmma)$ as a nonempty, connected
graph, of groups and injective group homomorphisms, in which every vertex is either
a sink, called a vertex group, or a source of valence two, called an edge group.  We shall use the
\textit{fundamental group} and the \textit{Bass-Serre tree} of $(\GG , \Gmma)$ as defined in~\cite{Ser77} and~\cite{DD89}.
\end{definitions}

Bass-Serre theory translates Theorem~\ref{th:tree} above into the following form.

\begin{theorem}\label{th:fund} Suppose that $G$ is the fundamental group of
a graph of groups $(\GG , \Gmma)$ such that the image of each  edge map
 \mbox{$\GG (e) \xrightarrow{\GG (e')} \GG (v)$}
is left relatively convex in its vertex group,~\mbox{$\GG (v)$}.
Then each  vertex group   is left relatively convex in $G$.
If   some  vertex  group   is  left orderable, then $G$ is
  left orderable.  Moreover, if the input orders are  given effectively,
 then the output orders are  given effectively.   \qed
\end{theorem}

\begin{remarks}  Theorem~\ref{th:fund} above generalizes the  result of  Chiswell
 that a
group is left orderable if it is the fundamental group of
a graph of groups such that each vertex group
is left ordered and each edge group is convex in each of its vertex groups;
see Corollary~3.5 of \cite{chiswell:ordered}.

The result of Chiswell is a consequence of Corollary 3.4 of \cite{chiswell:ordered},
which shows that a group is left orderable if it is the fundamental group of
a graph of groups such that each edge group is left orderable and each of its left orders extends
to a left order on each of its vertex groups.  (If, moreover,
 each edge group and vertex group is left ordered, and the maps from edge groups to
vertex groups respect the orders, then the fundamental group has a left order such that
the maps from the vertex groups to the fundamental group respect the orders.)
 This applies to the case of cyclic edge groups and left-orderable vertex groups.

Corollary 3.4 of \cite{chiswell:ordered} is, in turn,  a  consequence
of Chis\-well's   necessary and sufficient conditions for the fundamental group
of a graph of groups to be left orderable.  As his proof  involved ultraproducts, his
 orders were not constructed effectively.
\end{remarks}

\begin{example}\label{ex:A*B} Let  $A$ and $B$ be  groups, $C$ be a subgroup of $A$,
and \mbox{$x:C \to B$, $c \mapsto c^{\mkern3mu x},$}  be an injective homomorphism.
The graph of groups \mbox{$A \leftarrow C \rightarrow B$}, where the maps
are the inclusion map and $x$, has as fundamental group
\mbox{$A*_CB:= A*B/\normgen{\{ c^{-1}{\cdot}  c^{\mkern3mu x}
\mid c \in C\}},$}  called the
\textit{free product with amalgamation} with \textit{vertex groups $A$ and $B$}, \textit{edge group} $C$,
and \textit{edge map} $x$.
We then view $A$ and $B$ as subgroups of \mbox{$A*_CB$}.  In particular, $c^{\mkern3mu x}  = c$.

If $C$ is left relatively convex in each of $A$ and $B$, then $A$ and $B$ are
left relatively convex in \mbox{$A*_CB$}, by Theorem~\ref{th:fund} above.

In detail, suppose that \mbox{$G = A*_CB$}, that  $<_{\null_{\,\scriptstyle{A}}}$
is an  $A$-order  on $A/C$, and that \vspace{.5mm} $<_{\null_{\,\scriptstyle{B}}}$ is a  $B$-order
on~$B/C$.  The  Bass-Serre  left
$G$-tree $T$ for \mbox{$A \leftarrow C \rightarrow B$} has vertex-set $G/A \hskip2.5pt \dot{\cup} \hskip2.5pt G/B$
(where $\dot\cup$ denotes  the disjoint union) and edge-set $G/C$,
with  \mbox{$gA \xrightarrow{gC}{\mkern-9mu-}gB$}.
Then  $<_{\null_{\,\scriptstyle{A}}}$ and \vspace{.5mm}  $<_{\null_{\,\scriptstyle{B}}}$
determine a $G$-invariant local order on~$T$,
and we  have the associated   $G$-order $<_{\null_{\,\scriptstyle{T}}}$ on $\operatorname{V}T$,
as in Definitions~\ref{defs:trees} above.
Let us describe the   $G$-order $<_{\null_{\,\scriptstyle{T}}}$  on $G/A$.  Consider any $gA \in G/A$,
and write  \mbox{$gA = a_1b_1a_2b_2\cdots a_nb_nA$},  \mbox{$n \geq 0$}, where
 \mbox{$a_1 \in A$},   \mbox{$a_2,\ldots, a_n \in A{-}C$},  \mbox{$b_1,b_2,\ldots, b_n \in B{-}C$}.
 We then have a reduced $T$-path
\begin{align*}
 A  \xrightarrow{ a_1C }{\mkern-9mu-} \,  a_1B  \xrightarrow{ (a_1b_1C)^{-1} }{\mkern-9mu-} \,    a_1&b_1A
\xrightarrow{ a_1b_1a_2C }{\mkern-9mu-} \,  a_1b_1a_2B \xrightarrow{ (a_1b_1a_2b_2C)^{-1} }{\mkern-9mu-} \,  \ldots\\
& \hskip-59pt\ldots  \xrightarrow{ a_1b_1a_2b_2\cdots a_nC }{\mkern-9mu-} \, a_1b_1a_2b_2\cdots a_nB
 \xrightarrow{ (a_1b_1a_2b_2\cdots a_nb_nC)^{-1} }{\mkern-9mu-} \,  a_1b_1a_2b_2\cdots a_nb_nA = gA.
\end{align*}
The orientation-sum equals zero, and we have only the turn-sum, which simplifies  by the $G$-invari\-ance
of the local order to give\vspace{-2mm}
$$\textstyle\operatorname{sign}(A, <_{\null_{\,\scriptstyle{T}}}, gA) =
\operatorname{sign}  \Bigl(0,\,\, <_{\null_{\scriptstyle{\ZZ}}}, \,\,
 \sum\limits_{i=1}^n\operatorname{sign}(C, <_{\null_{\,\scriptstyle{B}}}, b_iC)
+ \sum\limits_{i=2}^{n}\operatorname{sign}(C ,<_{\null_{\,\scriptstyle{A}}}, a_iC)\Bigr).\vspace{-2mm}$$
\indent We record the case where $C=\{1\}$.
\end{example}

\begin{corollary}\label{cor:freefactors} In a  left-orderable group,
every free factor is left relatively convex.  \qed
\end{corollary}

\begin{example}\label{ex:ff} In a free group, every free factor is left relatively convex, by
Example~\ref{ex:trees} above.
\end{example}

\begin{example} Suppose that
 $A$ and $B$ are free groups, or, more generally, groups all of whose maximal  cyclic subgroups are
left relatively convex; see Corollary~\ref{cor:BH} above.
If $C$ is a maximal cyclic subgroup in both $A$ and $B$,
then $A$ and $B$  are left relatively convex in~$A*_CB$, by Example~\ref{ex:A*B} above.
\end{example}

\begin{example}\label{ex:HNN} Let  $A$ be a group, $C$ be a subgroup of $A$,
and \mbox{$x:C \to A$, $c \mapsto c^{\mkern3mu x},$}  be an injective homomorphism.
The graph of groups \mbox{$C \rightrightarrows A$}, where the maps
are the inclusion map and $x$, has as fundamental group
 \mbox{$A*_C\mkern1mu x:= A*\langle  x\mid \emptyset\rangle/\normgen{\{ x^{-1}{\cdot}c^{-1}{\cdot} x {\cdot}
c^{\mkern3mu x}
\mid c \in C\}}$},
 called the \textit{HNN extension}
 with \textit{vertex group}~$A$, \textit{edge group~$C$}, and \textit{edge map} $x$.
We  then view $A$ and \mbox{$\langle  x\mid \emptyset\rangle$} as  subgroups of  \mbox{$A*_C\mkern1mu x$}.  In particular, $c^{\mkern3mu x}
= x^{-1}cx$.

If $C$ and $C^{\mkern2mu x}$ are
left relatively convex in~$A$, then   $A$ is left relatively convex in $A*_C\mkern1mu x$,
 by Theorem~\ref{th:fund} above.

 If  \mbox{$G = A*_C\mkern1mu x$}, then the
 Bass-Serre  left $G$-tree $T$ for \mbox{$C \rightrightarrows A$} has vertex-set $G/A$ and edge-set $G/C$, with
 \mbox{$gA \xrightarrow{gC}{\mkern-9mu-}g x A$}.
\end{example}

\section{Surface groups and raags}\label{sec:examples}

 The following applies to all noncyclic surface groups.

\begin{example}\label{ex:HNN2}   Let \mbox{$G=\langle \mkern2mu  \{x\} \,\,\dot{\cup}\,\, \{y\}
   \,\,\dot{\cup}\,\,Z   \mid  x^{-1}y^{\mkern1mu \epsilon}xy  w \rangle$} with
\mbox{$\epsilon \in \{-1,1\}$} and
 \mbox{$w \in \langle \mkern2mu Z \mid \emptyset \,\rangle$}.
By Example~\ref{ex:ff} above, both\vspace{.5mm} $\langle\, y \,\rangle$ and $\langle\, y\,  w  \rangle$
are left relatively convex  in \vspace{.5mm}
$\langle \mkern2mu  \{y\} \,\,\dot{\cup}\,\, Z  \mid \emptyset \,\rangle$,
which in turn is left relatively convex in the HNN extension $G$,
by Example~\ref{ex:HNN} above. Here the Bass-Serre left $G$-tree $T$ has vertex-set
\mbox{$G/\langle \{y\}   \,\,{\cup} \,\, Z   \rangle$}
and edge-set \mbox{$G/\langle y \rangle$}.

 Notice that
\mbox{$\langle \mkern2mu  \{x\} \,\,{\cup} \,\, Z \, \rangle$} is not
left relatively convex in \mbox{$\langle \mkern2mu  \{x\} \,\,\dot{\cup}\,\, \{y\}
   \,\,\dot{\cup}\,\,Z   \mid  (xy)^2 =   x^{2}w^{-1}\rangle$}.

\end{example}

The following applies to all noncyclic surface groups except the Klein-bottle group.

\begin{proposition}   Let   \mbox{$G =   \langle \mkern2mu  \{x\}
 \,\,\dot{\cup}\,\, \{y\}    \,\,\dot{\cup}\,\,Z   \mid  [x,y]w \rangle$}  with
\mbox{$w \in \langle \mkern2mu Z \mid \emptyset \,\rangle$}.
  Then  every maximal  cyclic subgroup of $G$ is left relatively convex in~$G$.
\end{proposition}

\begin{proof}
Set  \mbox{$\bar G :=  G/\normgen{Z} \simeq \ZZ^2,$}  and let
 $G \to \bar G$, $g \mapsto \bar g$, denote the natural map.
We shall show that if $H$ is a  nonfree subgroup of $G$, then $\bar H \simeq \ZZ^2$, and
the result will then follow from Corollary~\ref{cor:BH} above.
Let $n \in \ZZ$.  It suffices to show that \mbox{$\bar H \cap \langle \bar x^{n} \bar y \rangle \ne \{1\}$}.
Let $\varphi$ denote the  automorphism of~$G$ such that  $\varphi(y) = x^{-n}y$ and $\varphi(z) = z$ for all $z \in \{x\} \cup Z$.
  By Example~\ref{ex:HNN2} above,  we have a left $G$-tree $T$
with free $G$-stabilizers  and edge-set~\mbox{$G/\langle y  \rangle$}.
By Bass-Serre theory, the nonfreeness of $\varphi(H)$ implies that the $\varphi(H)$-stabilizer
 of some $T$-edge $g_0\langle y \rangle$ must be nontrivial. Then
 \mbox{$\varphi(H) \cap \null^{g_0}\mkern-2mu\langle y \rangle \ne \{1\}$}.
Now   \mbox{$ H  \cap
\null^{\varphi^{-1}(g_0)}\mkern-2mu\langle x^ny \rangle
 = \varphi^{-1}(\varphi(H) \cap \null^{g_0}\mkern-2mu\langle y \rangle)\ne \{1\}$}.  It follows that
\mbox{$\bar H \cap \langle \bar x^{n} \bar y \rangle \ne \{1\}$}.
\end{proof}

\begin{corollary} In any surface group that is not the Klein-bottle group,  maximal
 cyclic subgroups are left relatively convex.\qed
\end{corollary}

\begin{definitions}\label{defs:raag}  Let $X$ be a set,  $R$ be a subset of  $[X,X]$  in
$\langle \mkern2mu X \mid \emptyset \,\rangle$, and \mbox{$G = \langle \mkern2mu X \mid R\,\rangle$}.
  We say that $G$ is
a \textit{right-angled Artin group},
or   \textit{raag} for short.  For example,   free groups and   free abelian groups are raags.

Let $Y$ be a subset of $X$.  The map \mbox{$X \to G$}
which acts as the identity map on $Y$ and sends \mbox{$X{-}Y$} to $\{1\}$ induces  well-defined
homomorphisms   \mbox{$G  \to G$} and \mbox{$G/\normgen{\mkern2muX{-}Y} \to G$}.
Moreover, the natural
composite \mbox{$G/\normgen{\mkern2muX{-}Y} \,\,\,\to\,\,\,G \,\,\,\to\,\,\, G/\normgen{\mkern2muX{-}Y}$}
 is the identity map, since it acts as such on
the generating set $Y$.
   Thus
we may identify  \mbox{$G/\normgen{\mkern2muX{-}Y}$}  with its image
$\langle Y \rangle$ in~$G$.  It follows that $\langle Y\rangle$ is a
raag.   We let $\pi_{\langle X \rangle \to \langle Y\rangle}$ denote
the  map  \mbox{$G \,\, \to \,\, G/\normgen{\mkern2muX{-}Y} = \langle Y \rangle$}.

For each \mbox{$x \in X$},     \mbox{$G = A\ast_Cx$} where
\mbox{$A = \langle X{-}\{x\} \rangle$},   \mbox{$C = \langle \{y \in X{-}\{x\} \mid [x,y] \in R^{\pm1} \}\rangle$}, and
\mbox{$x:C \to A$}, \mbox{$c\mapsto c^{\mkern2mux}$}, is the inclusion map.  In essence, this was noted by Bergman~\cite{Berg76}.

It is not difficult to show that $\langle\,  Y \,\rangle$ is left relatively
convex in $G$; since~(\ref{def:relconvex}.\ref{d1}) above
is a local condition, it suffices to verify this for $X$ finite,  and here it holds by induction on $|X|$ and
Example~\ref{ex:HNN} above.  In particular, $G$ is left orderable and, hence, torsion-free.
\end{definitions}

By \cite[Corollary 1.6]{AM15}, raags are {\nasmof} and therefore locally $2$-indicable. 
We do not know if raags are locally $n$-indicable for all positive integers $n$ or if they have the property that their maximal $n$-generated subgroups are left relatively convex. By Corollary~\ref{cor:nasmof}, we have  the following.

\begin{corollary}
Let $G$ be a subgroup a right-angled Artin group and $n$ a non-negative integer. Every maximal $n$-generated abelian subgroup of $G$ is left relatively convex in $G$. 
\end{corollary}

\section{Residually torsion-free nilpotent groups and left relative convexity}\label{sec:rtfn}

Corollary~\ref{cor:BH}, combined with the next few observations, provides many examples of left
relatively convex cyclic subgroups.

\begin{proposition}\label{p:nilpotent-2-indicable}
If $G$ is a finitely generated, nilpotent group with torsion-free center, then $G$ is 2-indicable.
\end{proposition}

\begin{proof}
Let $G$ be any group (not necessarily nilpotent or with torsion free center), $Z_1$ be its center,
and $Z_2$ be its second center, that is, $Z_2/Z_1$ is the center of $G/Z_1$.

For $g \in G$ and $a \in Z_2$, the commutator $[a,g]$ is in $Z_1$. From the identity
$[ab,g]=[a,g]^b[b,g]$, we obtain, for $a,b \in Z_2$, $[ab,g]=[a,g][b,g]$. Therefore, for any
element $g \in G$, $a \mapsto [a,g]$ is a homomorphism from $Z_2$ to $Z_1$, and $a \mapsto
([a,g])_{g \in G}$ is a homomorphism from $Z_2$ to $\prod_{g \in G} Z_1$ with kernel $Z_1$, which
implies that $Z_2/Z_1$ embeds into a power of $Z_1$.

We now let $G$ be a finitely generated, nilpotent group with torsion free center and we argue by
induction on the nilpotency class $c$ of G.

If $c = 0$, then $G$ is trivial, and hence 2-indicable. Assume that $c \ge 1$. Since $Z_2/Z_1$
embeds into a power of $Z_1$, which is a torsion-free group, $Z_2/Z_1$ itself is a torsion-free
group. Therefore $G/Z_1$ is a finitely generated, nilpotent group of class $c-1$ with torsion-free
center $Z_2/Z_1$. By the inductive hypothesis, $G/Z_1$ is 2-indicable. If $G/Z_1$ is noncyclic,
then $G/Z_1$ maps onto $\ZZ^2$, and so does $G$; thus we may assume that $G/Z_1$ is cyclic. In that
case, $G/Z_1$ is trivial, $G$ is a free abelian group, and, hence, $G$ is 2-indicable.
\end{proof}

\begin{remark}
Note that, under the assumption that $Z_1$ is torsion free, the observation that $Z_2/Z_1$ embeds
into some power of $Z_1$ yields that $Z_2/Z_1$, the center of $G/Z_1$, is itself torsion-free.
Inductive arguments then quickly yield that each upper central series factor $Z_{i+1}/Z_i$, for $i
\geq 0$, is torsion-free, each quotient $Z_j/Z_i$, for $j>i\geq0$, is torsion free, and under the
additional assumption that $G$ is nilpotent, each quotient $G/Z_i$, for $i \geq 0$, is
torsion-free; these are well-known results of Mal{\cprime}cev~\cite{malcev:tfn} and we could use
them to skip the first part in the proof of Proposition~\ref{p:nilpotent-2-indicable} and move
directly to the inductive part of the proof.

Proposition~\ref{p:nilpotent-2-indicable} also follows from Mal{\cprime}cev's result on quotients,
together with Lemma~13 in~\cite{bridson-b-e-s:ggrowth}, which states that every finitely generated,
nilpotent group that is not virtually cyclic maps onto $\ZZ^2$ (the proof of this result relies on
the fact that torsion-free, virtually abelian, nilpotent groups are abelian, which easily follows
from the uniqueness of roots in torsion-free nilpotent groups; another result of Mal{\cprime}cev
from~\cite{malcev:tfn}).

With all these choices before us, we still opted for the proof of
Proposition~\ref{p:nilpotent-2-indicable} provided above, because it is short and self-contained.
\end{remark}

\begin{proposition}\label{p:lrtfn}
Every locally residually torsion-free nilpotent group is locally 2-indicable.
\end{proposition}

\begin{proof}
Let $G$ be a locally residually torsion-free nilpotent group and $H$ a finitely generated subgroup
of $G$. Then $H$ is residually torsion-free nilpotent group. If $H$ has a noncyclic, torsion-free,
nilpotent quotient, then this quotient maps to $\ZZ^2$ by
Proposition~\ref{p:nilpotent-2-indicable}, and so does $H$. Otherwise, $H$ is residually-$\ZZ$,
which implies that it is abelian. Since $H$ is finitely generated and torsion-free, it is free
abelian, hence 2-indicable (in fact, $H$ is cyclic in this case, since we already excluded the
possibility of noncyclic quotients).
\end{proof}

\begin{remark}
Note that if $G$ is residually torsion-free nilpotent then it is also locally residually
torsion-free nilpotent. In particular, for finitely generated groups there is no difference between
being residually torsion-free nilpotent or being locally residually torsion-free nilpotent.
\end{remark}

\begin{example}\label{e:rtfn-examples}
If $G$ is a
\begin{itemize}
\item residually free group~\cite{magnus:free-embedding},
\item right-angled Artin group or a subgroup of a right-angled Artin group~\cite{droms:phd},
\item 1-relator group with presentation
\[
 \langle \ X,a,b \mid [a,b] = w \ \rangle,
\]
where $a,b \not \in X$ and $w$ is a group word over $X$, including fundamental groups of all
    compact surfaces other than the sphere, the projective plane, and the Klein
    bottle~\cite{baumslag:surface-groups,frederick:surface-groups,baumslag:nilpotent-reflections},
\item free group in any polynilpotent variety, including free solvable groups of any given
    class~\cite{gruenberg:residual}, or
\item pure braid group~\cite{falk-r:pure-residual},
\end{itemize}
then $G$ is a residually torsion-free nilpotent group.

By Proposition~\ref{p:lrtfn}, such a group $G$ is locally 2-indicable and, by
Corollary~\ref{cor:BH}, each maximal cyclic subgroup of $G$ is left relatively convex.
\end{example}

\end{document}